\documentclass[leqno,11pt]{amsart}

\usepackage{amsmath}
\usepackage{graphicx, color}
\usepackage{amscd}
\usepackage{amsfonts}
\usepackage{amssymb}
\usepackage{mathrsfs}
\usepackage{mathtools}
\usepackage{ulem}

\newtheorem{thm}{Theorem}[section]
\newtheorem{cor}[thm]{Corollary}
\newtheorem{lem}[thm]{Lemma}
\newtheorem{prop}[thm]{Proposition}

\newtheorem{rem}{Remark}[section]

\numberwithin{equation}{section}

\newcommand\be{\begin{equation}}
\newcommand\ee{\end{equation}}
\newcommand\R{\mathbb R}

\newcommand\rn{\mathbb R^n}
\newcommand\Rnp{\mathbb R^n_+}
\newcommand\rnp{\mathbb R^n_+}

\newcommand{\De}{\Delta}
\allowdisplaybreaks
\def\eps{\varepsilon}

\title[Liouville-type theorems in half-spaces]
{A Liouville-type theorem for the \\ Lane-Emden equation in a half-space}

\author[Dupaigne]{Louis Dupaigne}
\address{Institut Camille Jordan, UMR CNRS 5208, Universit\'e Claude Bernard Lyon 1,
69622 Vil\-leurbanne cedex, France}
\email{dupaigne@math.univ-lyon1.fr}

\author[Sirakov]{Boyan Sirakov}
\address{PUC-Rio, Departamento de Matematica \\
Rua Marqu\^es de S\~ao Vicente 225 \\
G\'avea, Rio de Janeiro - CEP 22451-900, Brazil}
\email{bsirakov@mat.puc-rio.br}

\author[Souplet]{Philippe Souplet}%
\address{Universit\'e Sorbonne Paris Nord,
CNRS UMR 7539, Laboratoire Analyse, G\'{e}om\'{e}trie et Applications,
93430 Villetaneuse, France}
\email{souplet@math.univ-paris13.fr}

\begin{document}

\begin{abstract} We prove that the Dirichlet problem for the Lane-Emden equation in a half-space has no positive solution
which is monotone in the normal direction.
As a consequence, this problem does not admit any positive classical solution which is bounded on finite strips.
This question has a long history and our result solves a long-standing open problem.
Such a nonexistence result was previously available only for bounded solutions,
or under a restriction on the power in the nonlinearity.
The result extends to general convex nonlinearities.

\smallskip
{\bf Keywords.} Lane-Emden equation, semilinear elliptic equation, half-space, Liouville-type theorem,
unbounded solutions.
\end{abstract}
\maketitle

\section{Introduction and main results}

In this paper we prove a Liouville type theorem for positive classical solutions of the Lane-Emden equation in a half-space
with Dirichlet boundary conditions.

Similarly to the original Liouville result on bounded harmonic functions, a Liouville type theorem states that a given PDE has no nontrivial solutions, and in most cases is restricted to signed solutions in the Euclidean space or some unbounded domain in that space. Arguably the most outstanding theorem of this type for semilinear elliptic equations was obtained in \cite{GS}, on the problem
\begin{equation}\label{eq-gs}
-\De u = u^p , \qquad u>0,
\end{equation}
where $p>1$. Gidas and Spruck proved that there do not exist classical
solutions of \eqref{eq-gs} in $\R^n$ provided $1<p<p_c:=(n+2)/(n-2)_+$
(see also \cite{ChL}, \cite{BVV} or \cite{QSbook} for
proofs of this result). If $p\ge p_c$ there are (bounded) positive solutions of \eqref{eq-gs}.

Equation \eqref{eq-gs}, usually referred to as the Lane-Emden equation, is generally viewed as the simplest and most representative semilinear elliptic equation, and as such has been the object of an enormous number of theoretical studies. For an extended and up-to-date list of references we refer to the recent book \cite{QSbook}. The Lane-Emden equation is also the base model for many more general equations, in terms of the elliptic operator in the left-hand side or the nonlinear function in the right-hand side of \eqref{eq-gs}.

A full answer to the existence question for \eqref{eq-gs} is currently not available for proper subdomains of $\rn$. This is so even for the {\it Dirichlet problem in the half-space}
$$\Rnp=\{x\in\R^n;\,x_n>0\},$$
despite its long history and the large number of works on that problem. Studying \eqref{eq-gs} in a half-space is important both because the half-space is the simplest unbounded domain with unbounded boundary and because performing a blow-up close to the boundary for general equations in a smooth domain leads to \eqref{eq-gs} in a half-space.

The latter observation, together with degree theory, was used by Gidas and Spruck in \cite{GS2} to prove existence results for a large class of elliptic equations in a smooth bounded domain. They proved that \eqref{eq-gs} with Dirichlet boundary condition has no solutions in a half-space provided $1<p\le p_c$.
The proof in \cite{GS2} uses a moving planes argument (combined with Kelvin transform), which reduces the problem to the one-dimensional case.

Ten years later Dancer \cite{D} proved by the method of
moving planes that bounded solutions of \eqref{eq-gs} with Dirichlet boundary condition in a half-space are monotone in
the normal direction, and deduced that a nontrivial bounded solution in $\rnp$ gives rise to a solution in $\R^{n-1}$; this implies
that the problem has no {\it bounded} solutions in the range $1<p<p_D:=(n+1)/(n-3)_+$.
A further development was made by Farina \cite{F}, who used variational estimates and stability to show that bounded
solutions (and even solutions that are stable outside a compact set in a slightly smaller range for $p$)
do not exist if $1<p<p_F(n):=(n^2-10n+8\sqrt{n}+13)/((n-3)(n-11)_+)$.
The most general results in these lines of research, as well as an extensive list of references, can be found in the recent work \cite{DF}.

The question of existence of bounded solutions of the Dirichlet problem for \eqref{eq-gs} in a half-space was fully answered a few years ago by Chen, Lin and Zou \cite{CLZ}, who proved that 
there are no such solutions for any $1<p<\infty$. In that paper the authors used a well-chosen auxiliary function involving derivatives of $u$, as well as convexity considerations.

To our knowledge, nonexistence of unbounded solutions of \eqref{eq-gs} in $\rnp$ is completely open
in the supercritical range $p>p_{c}$
 (and for unstable solutions for $p\ge p_F(n+1)$). Here we study this range, and prove that {\it for any} $p>1$
 there are no solutions  which are monotone in the $x_n$ direction. This has been generally expected and conjectured by the experts on Liouville type problems since the above mentioned work by Dancer (1991), which related monotonicity to nonexistence.

 A consequence of our result is the nonexistence of positive classical solutions which are bounded on finite strips.

\begin{thm}
\label{mainthmup}
Let $p>1$. Then the problem
\be\label{pbmDirichletHalfspace1}
\left\{\begin{array}{llll}
\hfill -\Delta u&=&u^p,& x\in\Rnp
\vspace{1mm} \\
\hfill u&=&0,& x\in\partial\Rnp
\end{array}
\right.
\ee
does not admit any positive classical solution which is monotone in the $x_n$ direction.
\end{thm}

 \begin{cor}\label{coro1}
For any $p>1$ problem \eqref{pbmDirichletHalfspace1} does not admit any 
positive classical solution
 which is bounded on finite strips, i.e.
$$\sup_{ \Sigma_R} u<\infty\ \hbox{ for each $R>0$, where }\ \Sigma_R=\{x\in \Rnp;\ x_n<R\}.$$
\end{cor}

Our method is not limited to the Lane-Emden equation. We actually show nonexistence for any convex nonlinearity $f$
with $f(0)=0$.
Theorem \ref{mainthmup} and Corollary \ref{coro1} are a special case
of the following nonexistence result for more general nonlinearities.

\begin{thm}
\label{mainthm7}
Let $f\in C^1([0,\infty))\cap C^2(0,\infty)$ be
 nonnegative and convex,  with $f(0)=0$ and $f\not\equiv0$.
 Let $u\ge0$ be a classical solution of 
\be\label{pbmDirichletHalfspace2}
\left\{\begin{array}{llll}
\hfill -\Delta u&=&f(u),& x\in\Rnp
\vspace{1mm} \\
\hfill u&=&0,& x\in\partial\Rnp
\end{array}
\right.
\ee
 which is monotone in the $x_n$ direction. Then $u\equiv0$. 

In particular  \eqref{pbmDirichletHalfspace2} does not admit any nontrivial nonnegative classical solution which is bounded on finite strips.
\end{thm}

\begin{rem}\label{rempos} That positive solutions of \eqref{pbmDirichletHalfspace2} which are bounded on finite strips are (strictly) monotone in the $x_n$-direction is a general fact, valid for every locally Lipschitz $f$ on $[0,\infty)$
such that $f(0)\ge0$. This follows from the Hopf lemma and the method of moving planes, in the form used in \cite{BCNAnn}. A full argument can be found in \cite[Theorem 3.1]{F2} or in the proof of
\cite[Theorem 3.1]{QScpde}. {Thus Corollary \ref{coro1} is an immediate consequence of Theorem \ref{mainthmup}.}

When the space dimension $n=2$, monotonicity holds even without the assumption of boundedness on strips, see \cite{DS}, \cite{FS}. Therefore, under the assumption of Theorem~\ref{mainthm7} with $n=2$, problem \eqref{pbmDirichletHalfspace2}
does not admit any positive classical solution at all, which gives an alternative proof of the result 
in \cite{GS2} for $n=2$ (other proofs can be found in \cite{FS}).
\end{rem}

\begin{rem}\label{remf0}
(i) If $f(0)>0$ and $f$ is convex, then the existence or nonexistence of nonnegative solutions
for problem \eqref{pbmDirichletHalfspace2} can be completely classified.
Since this follows from more or less known results and from rather simple arguments,
we treat this case in the appendix.

\smallskip

(ii) If $f(0)<0$, then the problem becomes very different, 
we refer to \cite{BCNAnn, FS} and the references therein for results on this case.
\end{rem}

The main point in the proof of Theorem~\ref{mainthm7} is to show that $u$ is convex in the $x_n$-direction, as was already observed in \cite{CLZ} for bounded solutions (see also \cite{QSbook} for a simplified proof). Indeed, our assumptions imply that $f$ grows linearly, that is, $f(u)\ge c_0u-c_1$ for some constants $c_0,c_1>0$. Since $u>0$, this implies in turn that the average $\overline u_{x,r}$ of $u$ over any ball $B(x,r)$ of fixed large radius is uniformly bounded above, see Lemma \ref{mainthm3a} below. Now, if $u_{x_n}>0$ and $u_{x_nx_n}\ge 0$, the average $\overline u_{x,r}$ cannot remain bounded as $x_n\to+\infty$.

To prove that $u_{x_nx_n}\ge0$, differentiate once the equation with respect to $x_n$. 
Then $v=u_{x_n}$ is a positive solution of the linearized equation
$$
-\Delta v = f'(u)v\qquad\text{in $\R^n_+$},
$$
and so $u$ is stable (cf.~Section~2). In other words, the maximum principle holds for the linearized operator $-\Delta -f'(u)$ on any bounded domain of $\R^n_+$.
Differentiate once more: since $f$ is convex, $w=u_{x_nx_n}$ solves
$$
-\Delta w -f'(u) w \ge 0 \qquad\text{in $\R^n_+$,}
$$
and moreover $w=0$ on $\partial\R^n_+$. If one could apply the maximum principle to $w$ in the whole of $\R^n_+$, we would be done. 
But of course such a statement cannot be true in general (as can be seen by simple examples such as the harmonic function $w=-x_n$; note that $f'\ge 0$).
However, we can follow a clever trick of \cite{CLZ} and observe that the function $z=(1+x_n)u_{x_n}$ is a strict subsolution of the linearized equation. More precisely,
$$
-\Delta z -f'(u)z = -2w\qquad\text{in $\R^n_+$,}
$$
so that the quotient
$$\xi=\frac wz=\frac{u_{x_nx_n}}{(1+x_n)u_{x_n}}$$
and the elliptic operator $\mathcal{L}={z^{-2}}\nabla\cdot(z^2\nabla)$ solve
$$
-\mathcal{L} \xi\ge  2\xi^2 \quad \mbox{ in }\;\Rnp
$$
(see Step 1 of the proof of Theorem~\ref{mainthm7} for details). 
In this reformulation of the problem, it remains to prove that $\xi\ge0$. The way we do that is completely different from the previous works and makes use of a new idea, based on some delicate estimates which we describe next.
The point is to use the full strength of the nonlinear right-hand side (as opposed to just its sign in \cite{CLZ}, \cite{QSbook}).
More precisely, we estimate the $L^\infty$ norm of $\xi_-$ 
by a kind of quantified Moser-type iteration,
rather than using a pointwise form of the maximum principle, which required 
the boundedness of $u$ in the previous works. Precisely, we use
test-functions given by powers of $\xi_-$ multiplied by powers of a standard cut-off and optimize the resulting inequality by relating the power of $\xi_-$ and the scaling parameter. The desired $L^\infty$ bound follows by sending
simultaneously the power and the scaling parameter to infinity.

For general convex nonlinearities this last step makes crucial use of a recent and difficult universal
$H^1$ estimate from \cite{CFRS}, valid for stable solutions of semilinear
elliptic equations on unit balls and half-balls. 
However, for a large class of nonlinearities, which includes in particular the Lane-Emden case $f(u)=u^p$ with $p>1$,
the $H^1$ bound can also be proved by simpler and more classical arguments (cf.~Lemma~\ref{lemstabilOUR}),
thus making the proof of Theorem~\ref{mainthmup} self-contained.

\medskip

The outline of the rest of the article is as follows.
Section~2 is concerned with stable solutions and gathers basic concepts and useful estimates.
In Section~3, we give two auxiliary lemmas that play a pivotal role in the proof of Theorem~\ref{mainthm7}.
The latter is proved in Section~4. Finally, the case $f(0)>0$ is treated in the appendix.

\begin{rem}\label{rempreprint} The present paper is an update and improvement of
part of an unpublished work by the last two authors (ArXiv preprint 2002.07247).
The Liouville type results therein were 
weaker, as they required growth restrictions on the solution as $x_n\to\infty$,
which are completely removed here.
\end{rem}

\section{ Estimates for stable solutions}

We start by recalling that, for any domain $\Omega\subset\R^n$ and $f\in C^1([0,\infty))$, a
nonnegative solution $u\in C^2(\Omega)$ of $-\Delta u=f(u)$ in $\Omega$ is said to be stable if
\be\label{defstabil}
\int_\Omega f'(u)\varphi^2\,dx\le \int_\Omega |\nabla\varphi|^2\,dx,
\quad\hbox{ for all $\varphi\in C^\infty_0(\Omega)$.}
\ee
In particular, monotone solutions are stable, namely:

\begin{prop}
\label{propstable}
Let $\Omega$ be any domain of $\R^n$, $f\in C^1([0,\infty))$ and let $u\in C^2(\Omega)$
be a nonnegative solution of $-\Delta u=f(u)$ in $\Omega$.
If $u_{x_n}>0$ in $\Omega$, then $u$ is stable.
\end{prop}

This fact is well known  (see e.g. \cite{Dup}) but we give a short proof for convenience.

\begin{proof} It is well 
known that if a linear elliptic operator admits a positive supersolution in the closure of a bounded domain,
then its first eigenvalue in this domain is positive (see for instance \cite{BNV}).

Since $v:=u_{x_n}>0$ satisfies $-\Delta v=f'(u)v$,
 the first Dirichlet eigenvalue $\lambda_\omega$ of the linearized operator is positive in any compactly included subdomain $\omega\subset\subset\Omega$. But by the variational characterization of the eigenvalue $\lambda_\omega$,
 $$\lambda_\omega=\inf_{0\not\equiv\varphi\in C^\infty_0(\omega)}
 \frac{\int_\omega (|\nabla\varphi|^2-f^\prime(u)\varphi^2)\,dx}{\int_\omega \varphi^2\,dx}>0$$
implies that $u$ is stable.
\end{proof}

Throughout the rest of the article, we shall use the following notation:
$$B_R=\{x\in \R^n;\ |x|<R\},\qquad B_R^+=B_R\cap \Rnp,\qquad \Sigma_R=\{x\in \Rnp;\ x_n<R\}$$
and $s_-=\max(-s,0)$ for all $s\in\R$.

The next two propositions give $H^1$ bounds for nonnegative stable solutions of
 \be\label{hyplemstab3}
 -\Delta u=f(u)\quad\hbox{ in $\Omega$.}
 \ee
We begin with a statement that holds in greater generality.
\begin{prop}
\label{lemstabilCFRS}
Let $f\in C^1([0,\infty)$ be nonnegative.
\smallskip

(i) Let $\Omega=B_{1/2}$, $\omega=B_{1/4}$ and let $u\in C^2(\Omega)$
 be a nonnegative stable solution of \eqref{hyplemstab3}.
 Then, we have
 \be\label{estimlemstab2}
\int_\omega |\nabla u|^2 \le C \Bigl(\int_\Omega u\Bigr)^2.
\ee
for some constant $C$ depending on $n$ only.

(ii) Let $\Omega=B_2^+$, $\omega=B_1^+$ and $\Gamma=B_2\cap\{x_n=0\}$
 and assume that $f$ is nondecreasing.
Let $u\in C^2(\Omega)$ be a nonnegative stable solution of \eqref{hyplemstab3} satisfying the boundary condition
\be\label{hyplemstab4}
 u\in C^1(\Omega\cup\Gamma),\qquad  u_{|\Gamma}=0.
 \ee
Then, estimate \eqref{estimlemstab2} holds true.
\end{prop}

Proposition~\ref{lemstabilCFRS} follows from \cite[Propositions~2.5 and 5.5]{CFRS} (after applying a simple dilation).
The proof is involved. Although this is not needed in our proof, a deep aspect of the result is that the constant in the right hand side
is independent of the nonlinearity $f$.

Since its proof is more elementary, we provide a second statement, which assumes at least power growth of $u$ and gives a uniform bound on the gradient of $u$ in $L^2$.

\begin{prop}
\label{lemstabilOUR}
Assume that $f\in C^1([0,\infty)$. Assume in addition that
\be\label{hyplemstab1}
 f(s)\ge cs^{1+\sigma}-K,\quad s\ge 0
\ee
 and
\be\label{hyplemstab2}
 sf'(s)\ge (1+\eta)f(s)-K,\quad s\ge 0,
\ee
for some constants $\eta,\sigma,c,K>0$.

\smallskip

(i) Let $\Omega=B_1$, $\omega=B_{1/2}$
and let $u\in C^2(\Omega)$ be a nonnegative stable solution of \eqref{hyplemstab3}.
Then we have
 \be\label{estimlemstab}
 \int_{\omega} |\nabla u|^2\, dx\le C(n,\eta,\sigma,c,K).
 \ee

(ii) Let $\Omega=B_2^+$, $\omega=B_1^+$ and $\Gamma=B_2\cap\{x_n=0\}$.
 Let $u\in C^2(\Omega)$ be a nonnegative stable solution of \eqref{hyplemstab3} satisfying the boundary condition \eqref{hyplemstab3}.
  Then estimate \eqref{estimlemstab} is true.
\end{prop}

Note that Proposition~\ref{lemstabilOUR} can be applied with the Lane-Emden nonlinearity $f(u)=u^p$, $p>1$, as well as $f(u)=e^u-1$.
The proof of Proposition~\ref{lemstabilOUR} is by now standard, see in particular~\cite{F}, and {simpler than that of Proposition~\ref{lemstabilCFRS}}.
We provide all the details for convenience of the reader.
\smallskip

\begin{proof}[Proof of Proposition~\ref{lemstabilOUR}]
Fix $q\ge 1$ and $\eps>0$ to be chosen below.
 Set $\rho=1$ (resp., $\rho=2$) in case of assertion (i) (resp., (ii)) and let $\varphi\in C^\infty_0(B_\rho)$ 
 be such that $\varphi=1$ on $B_{\rho/2}$
and $0\le \varphi\le 1$.
By density (using \eqref{hyplemstab4} in case of assertion (ii)), we may take $u\varphi^q$ as test-function in the stability inequality \eqref{defstabil}.
Denoting $\int=\int_\Omega$ and using Young's inequality, this yields
$$
\begin{aligned}
\int f'(u)u^2\varphi^{2q}
&\le \int|\nabla(u\varphi^{q})|^2
\le (1+\eps)\int \varphi^{2q}|\nabla u|^2 + C_\eps q^2\int u^2\varphi^{2q-2}|\nabla\varphi|^2
\end{aligned}$$
hence, by \eqref{hyplemstab2},
\be\label{lemstabeq1}
(1+\eta)\int f(u)u\varphi^{2q}
\le (1+\eps)\int \varphi^{2q}|\nabla u|^2 + C_\eps q^2\int u^2\varphi^{2q-2}|\nabla\varphi|^2
+K\int u\varphi^{2q}.
\ee
(Here and below $C_\eps$ denotes generic positive constants depending only on $\eps$.)
On the other hand, testing equation \eqref{hyplemstab3} with $u\varphi^{2q}$
(using \eqref{hyplemstab4} in case of assertion (ii)), we get
$$\int f(u)u\varphi^{2q} =\int\nabla u\cdot\nabla (u\varphi^{2q})
=\int\varphi^{2q}|\nabla u|^2 +2q\int u\varphi^{2q-1}\nabla u\cdot\nabla\varphi $$
hence, by Young's inequality,
$$\begin{aligned}
\int\varphi^{2q}|\nabla u|^2
&\le \int f(u)u\varphi^{2q} +\frac{\eps}{1+\eps}\int \varphi^{2q}|\nabla u|^2 +q^2C_\eps  \int u^2\varphi^{2q-2}|\nabla\varphi|^2,
\end{aligned}$$
which implies
\be\label{lemstabeq2}
\int\varphi^{2q}|\nabla u|^2
\le (1+\eps)\int f(u)u\varphi^{2q} +q^2C_\eps  \int u^2\varphi^{2q-2}|\nabla\varphi|^2.
\ee
Combining \eqref{lemstabeq1} and  \eqref{lemstabeq2}, we obtain
$$(1+\eta)\int f(u)u\varphi^{2q}
\le (1+\eps)^2\int f(u)u\varphi^{2q}+C_\eps q^2\int u^2\varphi^{2q-2}|\nabla\varphi|^2+K\int u\varphi^{2q}.$$
Taking $\eps=(1+\frac{\eta}{2})^{1/2}-1$ and using \eqref{hyplemstab2},
$0\le \varphi\le 1$, $|\nabla \varphi|\le C$, and Young's inequality, we obtain
$$\begin{aligned}
\frac{c\eta}{2}\int u^{2+\sigma}\varphi^{2q}
&\le \frac{\eta}{2}\int f(u)u\varphi^{2q} +\frac{\eta K}{2}\int u\varphi^{2q}
\le C q^2\int u^2\varphi^{2q-2}|\nabla\varphi|^2+C\int u\varphi^{2q} \\
&\le C+C(1+q^2) \int u^2\varphi^{2q-2}
\le \frac{c\eta}{4}\int u^{2+\sigma}\varphi^{(2+\sigma)(q-1)} +  C(1+q^2)^{\frac{2+\sigma}{\sigma}}.
\end{aligned}$$
Here, and in the rest of the proof, $C>0$ denotes a generic constant depending only on $\sigma,\eta,n,c,K$.
Now choosing $q=(2+\sigma)/\sigma$, so that $(2+\sigma)(q-1)=2q$, it follows from the last chain of inequalities that
$\int u^{2+\sigma}\varphi^{2q} \le C$
and then that
$$\int f(u)u\varphi^{2q} +\int u^2\varphi^{2q-2}|\nabla\varphi|^2\le C.$$
Going back to \eqref{lemstabeq2}, we deduce that $\int \varphi^{2q}|\nabla u|^2\le C$.
Since $\varphi=1$ on $B_{\rho/2}$, estimate \eqref{estimlemstab} follows.
\end{proof}

\section{Auxiliary results}

In this section we prove two lemmas which will be instrumental in the proof of Theorem~\ref{mainthm7}.
We start with our key test-function argument.

\begin{lem}
\label{lemdiv}
Consider the diffusion operator over $\Rnp$ given by
$$
\mathcal{L}=\frac1A\nabla\cdot(A\nabla \;),
$$
where $A\in L^\infty_{loc}(\overline{\Rnp})$, $A>0$ a.e. in $\Rnp$.
Let $d\mu=A\,dx$ denote the reversible invariant measure associated to $\mathcal{L}$.
Let $q>1$ and assume that $\xi\in H^1_{loc}\cap C(\overline{\Rnp})$ is a weak solution of
\be\label{pbmDirichletLemma}
-\mathcal{L}\xi\ge  (\xi_-)^q \quad \mbox{ in }\;\Rnp,
\ee
with $\xi\ge 0$ on $\partial\Rnp$.
Then, there exists a constant $C=C(n,q)$ such that for all $R>1$ and $m \ge \frac{q+1}{q-1}$, we have
\be\label{estimDirichletLemma}
\Vert \xi_{-}^{q-1}\Vert_{L^m(B_R^+,d\mu)}\le C\frac{m}{R^2} \bigl[\mu(B_{2R}^+\setminus B_R^+)\bigr]^{1/m}
\ee
\end{lem}

\begin{rem} \label{remauxil}
(i) Here a weak solution of \eqref{pbmDirichletLemma} is understood in the following sense
$$\int_{\Rnp} A (\xi_-)^q\phi \le \int_{\Rnp} A\nabla \xi\cdot \nabla\phi,$$
for all $\phi\in H^1_0(\Rnp)$ with $\phi\ge 0$ and ${\rm Supp}(\phi)\subset\subset\overline{\Rnp}$.
\smallskip

(ii) Estimate \eqref{estimDirichletLemma} is essentially optimal; see Remark~\ref{remauxil2} below.
\end{rem}

\begin{proof} Set $\theta=q-1>0$ and %
denote $\int=\int_{\Rnp}$ for simplicity. Fix $\alpha\ge 1$, and let $\varphi\in C^\infty(\overline{\Rnp})$ have compact support.
Since $\xi\ge0$ on $\partial\Rnp$, we may test \eqref{pbmDirichletLemma} with $\phi=(\xi_-)^{2\alpha-1}\varphi^2$.
 Using $\nabla \xi_-=-\chi_{\{\xi<0\}}\nabla \xi$ a.e., this yields
$$
\begin{aligned}
\int A(\xi_-)^{2\alpha+\theta}\varphi^2
&\le \int A\nabla \xi\cdot \nabla\bigl[(\xi_-)^{2\alpha-1}\varphi^2\bigr]\\
&= \int A(\xi_-)^{2\alpha-1}\nabla \xi\cdot \nabla(\varphi^2)
 -(2\alpha-1)\int A(\xi_-)^{2\alpha-2}|\nabla \xi_-|^2 \varphi^2.
\end{aligned}
$$
In addition, by Young's inequality, we have
$$
\begin{aligned}
\int A(\xi_-)^{2\alpha-1}\nabla \xi\cdot \nabla(\varphi^2)
&=-2\int A(\xi_-)^{2\alpha-1}\varphi (\nabla \xi_-)\cdot \nabla\varphi\\
&=-2\int A\sqrt{2\alpha-1}(\xi_-)^{\alpha-1}\varphi(\nabla \xi_-)\cdot \frac{(\xi_-)^\alpha\nabla\varphi}{\sqrt{2\alpha-1}}\\
&\le (2\alpha-1)\int A(\xi_-)^{2\alpha-2}\varphi^2|\nabla \xi_-|^2+\frac{1}{2\alpha-1}\int A (\xi_-)^{2\alpha}|\nabla\varphi|^2.
\end{aligned}
$$
Consequently,
$$\int A(\xi_-)^{2\alpha+\theta}\varphi^2
\le \frac{1}{2\alpha-1}\int A (\xi_-)^{2\alpha}|\nabla\varphi|^2.$$
Let us now choose $\varphi$ of the form $\varphi=\psi^m$, with $m\ge \frac{q+1}{q-1}$ and $0\le \psi\in C^\infty(\overline{\Rnp})$ with compact support.
By H\"older's inequality,
$$
\begin{aligned}
\int A(\xi_-)^{2\alpha+\theta} \psi^{2m}
&\le \frac{m^2}{2\alpha-1}\int A (\xi_-)^{2\alpha}|\nabla\psi|^2\psi^{2m-2} \\
&\le \frac{m^2}{2\alpha-1}\Bigl(\int A (\xi_-)^{2\alpha+\theta}\psi^{\frac{(m-1)(2\alpha+\theta)}{\alpha}}\Bigr)^{\frac{2\alpha}{2\alpha+\theta}}
\Bigl(\int A |\nabla\psi|^{\frac{2(2\alpha+\theta)}{\theta}}\Bigr)^{\frac{\theta}{2\alpha+\theta}}.
\end{aligned}
$$
 Now we choose $\alpha=\theta(m-1)/2\ge 1$, hence $2\alpha+\theta=\theta m$ and $\frac{(m-1)(2\alpha+\theta)}{\alpha}=2m$, so we get
$$\Bigl(\int A(\xi_-)^{\theta m} \psi^{2m}\Bigr)^{\frac{1}{m}}\le \frac{m^2}{\theta(m-1)-1}\Bigl(\int A |\nabla\psi|^{2m}\Bigr)^{\frac{1}{m}}.$$
Fix a radially symmetric function $\psi_1\in C^\infty_0(\R^n)$ such that $\psi_1\ge 0$, with $\psi_1(y)=1$ for $|y|\le 1$ and $\psi_1(y)=0$ for $|y|\ge 2$.
Let $R>1$ and set in the last inequality $\psi(x)=\psi_1(x/R)$ for $x\in \overline{\rnp}$. We obtain
$$\Bigl(\int_{B_R^+} A(\xi_-)^{m(q-1)}\Bigr)^{\frac{1}{m}}\le {\frac{\|\nabla\psi_1\|_\infty^2}{R^2}}
\frac{m^2}{m(q-1)-q}\Bigl(\int_{B_{2R}^+\setminus B_R^+} A \Bigr)^{\frac{1}{m}},$$
and \eqref{estimDirichletLemma} follows.
\end{proof}

The second lemma provides a basic local $L^1$ estimate for the solution
and guarantees the nonexistence of solutions that are monotone and convex in the $x_n$ direction.

\begin{lem}
\label{mainthm3a}
Let $f:(0,\infty)\to \R$ be such that
\be\label{hypf2}
f(s)\ge c_0 s-c_1,\quad s>0,
\ee
for some $c_0, c_1>0$. Then:

\vskip 1pt

(i) There exists $R_0=R_0(c_0,n)>0$
such that any
positive supersolution of $-\Delta u =f(u)$ in~$\overline B_{2R_0}(x_0)$ satisfies
$$\int_{B_{R_0}(x_0)}u(x)\,dx\le C=C(c_0,c_1,n).$$

(ii) Let $R_0$ be given by assertion (i) and let $u$ be a positive supersolution of problem \eqref{pbmDirichletHalfspace2}
such that $u_{x_n}\ge 0$.
Then, for every $x_0\in \overline\Rnp$, we have
$$\int_{B_{R_0}(x_0)\cap \Rnp}u(x)\,dx\le C=C(c_0,c_1,n).$$

(iii) Problem \eqref{pbmDirichletHalfspace2}
does not admit any supersolution $u\in C^2(\overline\Rnp)$ such that $u_{x_n}>0$ and $u_{x_nx_n}\ge 0$ in $\overline\Rnp$.
\end{lem}

 \begin{rem} If $f(s)/s\to \infty$ as $s\to \infty$, the $L^1$-bound in (ii) remains true even if the hypothesis $u_{x_n}\ge0$ does not hold, see Theorem 3 in \cite{S2}.
\end{rem} 

\begin{proof}
(i) Take $R_0>0$ so that the principal
eigenvalue $\lambda_1(2R_0)$ of the Dirichlet Laplacian on $B_{2R_0}$ satisfies
$\lambda_1(2R_0)=c(n)(2R_0)^{-2}=c_0/2$.
The conclusion then follows by using assumption \eqref{hypf2} and testing $-\Delta u =f(u)$ in $B_{2R_0}$ with the corresponding first eigenfunction $\varphi_1$, which is such that $\varphi_1\ge \delta(R_0)>0$ in $B_{R_0}$, by the Harnack inequality.
\smallskip

(ii) This follows from assertion (i) by noting that, owing to $u_{x_n}\ge 0$,
$$\int_{B_{R_0}(x_0)\cap \Rnp}u(x)\,dx\le \int_{B_{R_0}(2R_0e_n+x_0)}u(x)\,dx \le C(c_0,c_1,n).$$

(iii) Assume for contradiction that $u\in C^2(\overline\Rnp)$ is a solution
such that $u_{x_n}>0$ and $u_{x_nx_n}\ge 0$ in $\overline\Rnp$. Let $R_0$ be the number from (i).
Our assumptions imply that
$$\eta:=\inf \bigl\{ u_{x_n}(x',0);\ x'\in\R^{n-1},\ |x'|<R_0\bigr\}>0.$$
Since $u_{x_nx_n}\ge 0$, it follows that
$$u(x',y)\ge \eta y\quad\hbox{ for all $y>0$ and $x'\in\R^{n-1}$ such that $|x'|<R_0$.}$$
Hence by (i), for all $s>R_0$,
$$C\ge \int_{B_{R_0}(se_n)}u(x)\,dx \ge \omega_nR_0^n\eta (s-R_0),\quad s>R_0.$$
But this is a contradiction, for sufficiently large $s$.
\end{proof}

\section{Proof of Theorems \ref{mainthmup} and \ref{mainthm7}}

Theorem \ref{mainthmup} is a special case of Theorem \ref{mainthm7}.
The proof of Theorem \ref{mainthm7} is done in three steps.

\begin{proof}[Proof of Theorem \ref{mainthm7}]

\

\smallskip

{\bf Step~1.} {\it Preliminaries.}
Since $f\in C^1([0,\infty))\cap C^2(0,\infty)$, by elliptic regularity we have
$u\in W^{3,s}_{loc}(\overline{\Rnp})\cap W^{4,s}_{loc}(\Rnp)$ for all $s\in(1,\infty)$.
Set
$$v:=u_{x_n}.$$
Since $v\ge 0$ by assumption, we have $v>0$ on $\partial\Rnp$ by the Hopf lemma
and $v>0$ in $\Rnp$ by the strong maximum principle applied to the equation
$$-\Delta v=f'(u)v.$$
We set
$$w:=u_{x_nx_n},\quad z=(1+x_n)v,\quad x\in \R^n,$$
and define the 
auxiliary function
\be\label{def-xi}
\xi:=\frac{w}{z}=\frac{u_{x_nx_n}}{(1+x_n)u_{x_n}},\qquad x\in \overline{\Rnp}.
\ee

{\bf Step~2.} {\it Elliptic inequality for $\xi$.}
We first claim that $\xi$ is a (strong, hence weak) solution of
\be\label{S-CLZeqnxi}
-\nabla\cdot(z^2\nabla \xi)\ge 2z^2\xi^2,\qquad x\in \Rnp 
\ee
and
\be\label{S-CLZeqnxi2}
\xi=0,\quad x\in\partial\Rnp.
\ee

Indeed, we have
$$z^2\nabla \xi=z^2\nabla\Bigl(\frac{w}{z}\Bigr)
=z\nabla w-w\nabla z,$$
hence
$$-\nabla\cdot(z^2\nabla \xi)=w\Delta z-z\Delta w.$$
Since, on the other hand,
$$-\Delta w=f'(u)w+f''(u)v^2$$
and
$$-\Delta z=-(1+x_n)\Delta v-2v_{x_n}=f'(u)z-2w,$$
we deduce that
$$
\begin{aligned}
-\nabla\cdot(z^2\nabla \xi)
&=w(-f'(u)z+2w)+z(f'(u)w+f''(u)v^2)\\
\noalign{\vskip 1mm}
&=2w^2+f''(u)zv^2=2z^2\xi^2+(1+x_n)f''(u)v^3.
\end{aligned}
$$
Since $f''\ge 0$ and $v>0$, inequality \eqref{S-CLZeqnxi} follows.

As for \eqref{S-CLZeqnxi2}, it is a consequence of $u=0$ and $u_{x_nx_n}=\Delta u=-f(0)=0$ on $\partial\Rnp$.

\smallskip
{\bf Step~3.} {\it Convexity in the normal direction and conclusion}.
We claim that
\be\label{S-uxnxn}
u_{x_nx_n}\ge0 \qquad \mbox{in }\;\rnp.
\ee

First of all, it follows from \eqref{S-CLZeqnxi}, \eqref{S-CLZeqnxi2} and Lemma~\ref{lemdiv}
with $q=2$ and $A=z^2$ that, for all $R>1$ and $m \ge 3$,
\be\label{estimLemmaLouis}
\Bigl(\int_{B_R^+} z^2(\xi_-)^m\Bigr)^{1/m}\le C(n) {m\over R^2}\Bigl(\int_{B_{2R}^+\setminus B_R^+} z^2\Bigr)^{1/m}.
\ee

We proceed to estimate the right-hand side of \eqref{estimLemmaLouis}.
By Proposition~\ref{propstable},
$u$ is a stable solution on any subdomain of $\Omega\subset\Rnp$.
 In addition, the hypothesis \eqref{hypf2} of Lemma~\ref{mainthm3a} is clearly
satisfied by any nonnegative convex function such that  $f(0)=0$ and $f\not\equiv0$.
For instance, if $f(t_0)>0$ for some $t_0>0$, we have
$$
f(t)\ge \frac{f(t_0)}{t_0}\, t- f(t_0), \quad t>0,
$$
because  $(f(t)-f(0))/t$ is nondecreasing in $t$.
It then follows from  Proposition~\ref{lemstabilCFRS} and Lemma~\ref{mainthm3a}(i)-(ii) that 
\be\label{intestimate2a}
\int_{\omega} |\nabla u|^2 \le C(n,f),
\ee
for any set $\omega$ which is either of the form $B_{1/4}(a)$ with
$a\in\Rnp\setminus \Sigma_{1/2}$ or $B_{1}(a)\cap \Rnp$ with $a\in\partial\Rnp$.

We note that in the case of Theorem~\ref{mainthmup}, namely $f(u)=u^p$ with $p>1$
(or more generally for $f$ satisfying assumptions \eqref{hyplemstab1}-\eqref{hyplemstab2}), 
estimate \eqref{intestimate2a} follows from Proposition~\ref{lemstabilOUR} above, and the more difficult Proposition~\ref{lemstabilCFRS}
is not needed.

Now, for any $R>1$, $B_{2R}^+\setminus \Sigma_{1/2}$ can be covered by at most $C(n)R^n$ balls
of radius $1/4$ and centered in $\Rnp\setminus \Sigma_{1/2}$,
while
$B_{2R}^+\cap\Sigma_{1/2}$ can be covered by at most $C(n)R^{n-1}$ sets of the form $B_{1}(a)\cap \Rnp$,
with $a\in\partial\Rnp$.
Consequently,
$$
\begin{aligned}
\int_{B_{2R}^+} z^2
&=\int_{B_{2R}^+} (1+x_n)^2(u_{x_n})^2 \\
&\le (1+2R)^2 \int_{B_{2R}^+\setminus \Sigma_{1/2}} |\nabla u|^2
+\frac94\int_{B_{2R}^+\cap\Sigma_{1/2}} |\nabla u|^2 \\
&\le C(n,f)(R^{n+2} + R^{n-1})\le CR^{n+2},\quad R>1.
\end{aligned}$$
For any $m\ge 3$, setting $R=m$ in \eqref{estimLemmaLouis}, we get for all $D\ge 3$ and $m\ge D$ that
\be\label{intestimate3}
\Bigl(\int_{B_D^+} z^2(\xi_-)^m\Bigr)^{1/m}\le {Cm^{(n+2)/m}\over m}\to 0,\quad m\to \infty.
\ee
But since $z>0$ in $\overline{B_D}$, for each $D\ge 3$ the left-hand side of \eqref{intestimate3} converges to $\|\xi_-\|_{L^\infty(B_D^+)}$ as $m\to\infty$.
It follows that $\xi\ge 0$ in $\Rnp$, hence \eqref{S-uxnxn} holds.
 In view of \eqref{S-uxnxn}, Theorem~\ref{mainthm7} follows from Lemma~\ref{mainthm3a}(iii).
\end{proof}

\begin{rem} \label{remauxil2}
In the last part of the proof of Theorem~\ref{mainthm7}, we have 
combined Lemma~\ref{lemdiv} with the available information of polynomial growth of the measure associated with the weight $A=z^2$. 
Under the hypotheses of Lemma~\ref{lemdiv}, one can actually show a maximum principle 
assuming a weaker, almost Gaussian, growth condition for the weight $A$.
Namely, if 
\be\label{growthcondmu}
\log\bigl[\mu(B_R^+)\bigr]=o(R^2),\quad R\to\infty,
\ee
then any weak solution $\xi$ of \eqref{pbmDirichletLemma} with $\xi\ge 0$ on $\partial\Rnp$ satisfies $\xi\ge 0$ a.e.~in $\Rnp$.
This follows by taking $m=\eps R^2$ in estimate \eqref{estimDirichletLemma} (for each given $\eps>0$) and letting $R\to\infty$.
Moreover, assumption \eqref{growthcondmu} is essentially optimal, as shown by the counter-example:
$$A(x)=\exp\bigl[(x_n)^k\bigr],\quad \xi=-x_n, \quad\hbox{with $k>2$ and $q=k-1$}.$$
This counter-example also shows the optimality of estimate \eqref{estimDirichletLemma}: taking $m=R^k$, we see that
both sides of \eqref{estimDirichletLemma} behave like $R^{k-2}$ for $R\gg 1$ up to multiplicative constants.
Condition~\eqref{growthcondmu} has a natural geometric interpretation as a volume growth condition on $\Rnp$
equipped with the metrics associated with the operator $\mathcal{L}$.
 Lemma~\ref{lemdiv} can be extended to diffusion operators in more general settings.
 These ideas and their applications will be taken up elsewhere.
\end{rem}

\section{Appendix: the case $f(0)>0$.}

Let $f:[0,\infty)\to \R$ be a convex function such that $f(0)>0$.
The purpose of this appendix is to show that the Liouville property
(i.e., the nonexistence of nontrivial nonnegative solutions)
for problem \eqref{pbmDirichletHalfspace2} can be completely classified.
Some of the results below do not actually require the convexity of~$f$ and/or are valid for supersolutions.

\smallskip

First of all, the following simple result allows us to reduce the problem to the case $f>0$.

\begin{prop} \label{propappendix1}
Let $f:[0,\infty)\to \R$ be continuous and satisfy $f(0)>0$ and $f(a)=0$ for some $a>0$.
Then there exists a positive, bounded, increasing solution $u\in C^2([0,\infty))$ of the one-dimensional problem
\be\label{pbmODEHalfspace}
\left\{\begin{array}{llll}
\hfill -u''&=&f(u),&\quad x>0
\vspace{1mm} \\
\hfill u(0)&=&0.
\end{array}
\right.
\ee
\end{prop}

\begin{proof}
Assume without loss of generality that $f>0$ on $[0,a)$ and set $F(t)=\int_0^t f(s)ds$.
Let $u$ be the unique  solution of the initial value problem $-u''=f(u)$ ($x>0$) with initial data
$u(0)=0$ and $u'(0)=\sqrt{2F(a)}$, defined on its maximal interval of existence.
Multiplying the equation by $u'$, we obtain the conservation of the energy $E(u):={u'}^2+2F(u)$, hence
\be\label{EnergyODE}
{u'}^2=2(F(a)-F(u)).
\ee
The function $t\to F(a)-F(t)$ is a decreasing bijection from $[0,a]$ to $[F(a),0]$.
Therefore, the  right hand side of \eqref{EnergyODE}, which is positive for $x>0$ small, can never vanish, 
since otherwise at its first zero
we would have $u=a$ and $u'=0$, contradicting local uniqueness.
It follows that $u$ exists for all $x>0$ and satisfies $0\le u<a$ and $u'>0$.
The proposition is proved.
\end{proof}

Now, for the case of positive nonlinearities, the following result (see \cite[Corollary~5.6]{AS}),
actually valid for supersolutions, yields an almost sharp condition for
the nonexistence of nonnegative solutions for problem~\eqref{pbmDirichletHalfspace2}.

\begin{prop} \label{propappendix2}
Let $f:[0,\infty)\to (0,\infty)$ be continuous and satisfy
\be\label{Condfinfty}
f(s)\ge cs^{-1},\quad s\ge 1,
\ee
for some constant $c>0$.
Then the inequality
\be\label{pbmSupersolHalfspace}
-\Delta v\ge f(v),\quad x\in\Rnp
\ee
admits no nonnegative solution $v\in C^2(\Rnp)$.
\end{prop}

The sharpness of assumption \eqref{Condfinfty} for problem~\eqref{pbmDirichletHalfspace2}
can be seen from the example
$$f(u)=(1+u)^{-q},\quad\hbox{ with } u(x)=(1+kx)^{2/(q+1)}-1,\ x>0,$$
with $q>1$ and $k=(q+1)^{-1}\sqrt{2(q-1)}$.\smallskip

If we assume additionally that $f$ is monotone, 
assumption \eqref{Condfinfty} can be weakened,
giving rise to an optimal integrability condition.

\begin{prop} \label{propappendix3}
Let $f:[0,\infty)\to (0,\infty)$ be continuous.

\smallskip

(i) If
\be\label{Condfintegrability1}
\int_0^\infty f(s)\, ds<\infty
\ee
then there exists a positive solution $u\in C^2([0,\infty))$ of the one-dimensional problem
\eqref{pbmODEHalfspace}.

(ii) Assume that $f$ is nonincreasing and that
\be\label{Condfintegrability2}
\int_0^\infty f(s)\, ds=\infty.
\ee
Then inequality \eqref{pbmSupersolHalfspace}
admits no nonnegative solution $v\in C^2(\Rnp)$.
\end{prop}

\begin{proof}
{\bf Step 1.} We first show that the solvability on $[0,\infty)$ of the one-dimensional problem \eqref{pbmODEHalfspace} is equivalent to \eqref{Condfintegrability1}.
Assume \eqref{Condfintegrability1} is satisfied.
Then we claim that a solution of \eqref{pbmODEHalfspace} is given by
$$u(y)=h^{-1}(y),$$
where
$$h(t)= \int_0^{t} dz/g(z),\quad g(z):=\Bigl(2\int_z^\infty f(s)\,ds\Bigr)^{1/2}.$$
(Note that $h$ is an increasing bijection from $[0,\infty)$ to $[0,\infty)$, with $h(t)\gg t$ at infinity.)

\smallskip
Indeed, we have $u(0)=h^{-1}(0)=0$ and $u$ satisfies
$$
\begin{aligned}
h(u(y))=y\
&\Longrightarrow u'(y)={1\over h'(u(y))}=g(u(y)) \\
&\Longrightarrow {u'}^2(y)=2\int_{u(y)}^\infty f(s)\,ds
\Longrightarrow 2[u'u''](y)+2f(u(y))u'(y).
\end{aligned}
$$
Since $u'>0$, it follows that $-u''(y)=f(u(y))$.
\smallskip

Next consider the case \eqref{Condfintegrability2} and assume for contradiction that
\eqref{pbmODEHalfspace} admits a solution $u>0$ defined on $[0,\infty)$.
Then $u$ is concave, hence $\ell:=\lim_{y\to\infty} u'(y)$
exists, with $\ell\in [0,\infty)$ (since $u>0$), and $u$ is nondecreasing.
Set $F(z)=\int_0^z f(s)\,ds$. By the energy relation
$${u'}^2+2F(u)=C:={u'}^2(0),$$
we see that $F(u)$ is bounded and, by \eqref{Condfintegrability2}, we deduce that
$a:=\lim_{y\to\infty} u(y)<\infty.$
Therefore
$$\inf_{y>0} f(u(y))\ge \sigma=\inf_{[0,a]} f>0,$$
hence
$$u''\le -\sigma,\quad y>0.$$
But this contradicts $u>0$.

\smallskip

{\bf Step 2.} Assume that $f$ is nonincreasing and that
 \eqref{pbmSupersolHalfspace} admits a nonnegative solution $v\in C^2(\Rnp)$,
We claim that \eqref{pbmODEHalfspace} also has a solution defined on $[0,\infty)$.
This will complete the proof of Proposition \ref{propappendix3}.

\smallskip
 To prove the claim, we shall use a simplified version of some arguments from~\cite{GQS}.
By shifting in the $x_n$-direction, we can assume
$v\in C^2(\overline{\Rnp})$ and $v>0$ in $\overline{\Rnp}$.
Consider the following sequence of problems in cylinders:
\be\label{pbmDirichletCylinder}
\left\{\begin{array}{llll}
\hfill -\Delta u_j&=&f(u_j),& x\in\Omega_j
\vspace{1mm} \\
\hfill u_j&=&0,& x\in\partial\Omega_j
\end{array}
\right.
\ee
where $j$ is a positive integer and
$$\Omega_j=\bigl\{x=(x',x_n)\in \Rnp;\ |x'|<j,\ x_n<j\bigr\}.$$
It follows from standard monotone iteration that \eqref{pbmDirichletCylinder} admits a nonnegative classical solution
$u_j\in C^2(\Omega_j)\cap C(\overline{\Omega_j})$.
Indeed, it suffices to use $\underline u=0$ as a (strict) subsolution and, as a supersolution,
the unique positive solution $\overline u$ of the linear problem $-\Delta \overline u=f(0)$ in $\Omega_j$
with Dirichlet boundary conditions
(note that the latter satifies $-\Delta \overline u\ge f(\overline u)$ since $f$ is nonincreasing).
Next, since $f$ is nonincreasing, the comparison principle is 
valid for the  nonlinear operator $\Delta\cdot + f(\cdot)$, in any bounded domain.
Consequently, we get $0\le u_j\le v$ in $\Omega_j$. This, along with elliptic estimates,
guarantees that some subsequence of $u_j$ converges, locally uniformly in $\overline{\Rnp}$,
to a positive solution $u\in C^2(\Rnp)\cap C(\overline{\Rnp})$ of \eqref{pbmDirichletHalfspace2}.

Finally, for any fixed $e\in\R^{n-1}$, consider the horizontal translation $u_{j,e}(x',x_n)=u_j(x'+e,x_n)$,
defined in the domain $\Omega_{j,e}=\bigl\{x\in \Rnp;\ |x'-e|<j,\ x_n<j\bigr\}$.
Applying the comparison principle in $\Omega_{j,e}$, we obtain 
$u_{j,e}\le u$ in $\Omega_{j,e}$.
Passing to the limit $j\to\infty$, we get $u(x'+e,x_n)\le u(x',x_n)$ for all $(x',x_n)\in \Rnp$.
Since $e\in\R^{n-1}$ is arbitrary, we see that $u$ depends only on $x_n$
and is hence a positive solution of \eqref{pbmODEHalfspace}.
\end{proof}

Observe that if $f$ is convex and $f(0)>0$,
Propositions~\ref{propappendix1}-\ref{propappendix3} give a complete classication
for the Liouville property in problem \eqref{pbmDirichletHalfspace2}. Indeed,
either $f$ vanishes somewhere and solutions exist by Proposition \ref{propappendix1}, or $f$ is positive and nonincreasing and then Proposition \ref{propappendix3} gives a necessary and sufficient condition for existence,
or $f$ is positive and nondecreasing in $(a,\infty)$ for some $a\ge0$ and then by Proposition \ref{propappendix2} there are no solutions.
\bigskip

{\bf Acknowledgements.}
The second author is funded by CNPq grant 427056/2018-7, and FAPERJ grant E-26/203.015/2017.
This work was partially done during a visit of the third author to the Mathematics Departement of the
Pontifícia Universidade Cat\'olica do Rio de Janeiro. He thanks this institution for the hospitality.

\end{document}